\documentclass{amsart}

\usepackage{amssymb,amsmath,latexsym,amsthm}
\usepackage{MnSymbol}
\usepackage{graphicx}
\usepackage{fontenc}%
\usepackage{mathrsfs}
\usepackage{color}

\usepackage[verbose]{wrapfig}
\usepackage[english]{babel} 
\usepackage{pstricks-add} 

\newtheorem{theorem}{Theorem}[section]
\newtheorem{definition}[theorem]{Definition}
\newtheorem{lemma}[theorem]{Lemma}

\newtheorem{example}[theorem]{Example}

\newtheorem{remark}[theorem]{Remark}

\usepackage{pstricks,pst-text,pst-grad,pst-node,pst-3dplot,pstricks-add,pst-poly}
\definecolor{pink}{rgb}{1, .75, .8}
\psset{arrows=->, labelsep=3pt, mnode=circle}

\definecolor{lgrey}{gray}{.85}

\def\defineTColor#1#2{%
 \newpsstyle{#1}{%
  fillstyle=vlines,hatchcolor=#2,
  hatchwidth=0.1\pslinewidth,
  hatchsep=1\pslinewidth}%
  }
\defineTColor{Tgray}{gray}
\defineTColor{Tgreen}{green}
\defineTColor{Tyellow}{yellow}
\defineTColor{Tred}{red} 
\defineTColor{Tblue}{blue} 
\defineTColor{Tmagenta}{magenta}
\defineTColor{Torange}{orange}
\defineTColor{Twhite}{white}
\defineTColor{Tgray}{lgrey}
\defineTColor{Tbgreen}{brightgreen}

\newcommand{\sn}{\mathop{\delta}\limits^{\doublewedge}}
\newcommand{\CL}{\mbox{CL}}
\newcommand{\cl}{\mbox{cl}}
\newcommand{\Int}{\mbox{int}}
\begin{document}

\title[Strongly near proximity \& hyperspace topology]{Strongly near proximity\\ \& hyperspace topology}

\author[J.F. Peters]{J.F. Peters$^{\alpha}$}
\email{James.Peters3@umanitoba.ca, cguadagni@unisa.it}
\address{\llap{$^{\alpha}$\,}Computational Intelligence Laboratory,
University of Manitoba, WPG, MB, R3T 5V6, Canada and
Department of Mathematics, Faculty of Arts and Sciences, Ad\i yaman University, 02040 Ad\i yaman, Turkey}
\author[C. Guadagni]{C. Guadagni$^{\beta}$}
\address{\llap{$^{\beta}$\,}Computational Intelligence Laboratory,
University of Manitoba, WPG, MB, R3T 5V6, Canada and
Department of Mathematics, University of Salerno, via Giovanni Paolo II 132, 84084 Fisciano, Salerno , Italy}
\thanks{The research has been supported by the Natural Sciences \&
Engineering Research Council of Canada (NSERC) discovery grant 185986.}

\subjclass[2010]{Primary 54E05 (Proximity); Secondary 54B20 (Hyperspaces)}

\date{}

\dedicatory{Dedicated to the Memory of Som Naimpally}

\begin{abstract}
This article introduces strongly near proximity $\sn$, which represents a new kind of proximity called \emph{almost proximity}.  A main result in this paper is the introduction of a hit-and-miss topology on $\mbox{CL}(X)$, the hyperspace of nonempty closed subsets of $X$, based on the strongly near proximity.
\end{abstract}

\keywords{Hit-and-Miss Topology, Hyperspaces, Proximity, Strongly Far, Strongly Near}

\maketitle

\section{Introduction}
Recently, we introduced \emph{strongly far proximities}~\cite{PetersGuadagni2015stronglyFar}.  Strongly far proximities make it possible to distinguish between a weaker form of farness and a stronger one associated with the \emph{Efremovi\v c property}. In this article, we introduce a framework that enables us to single out when two sets are \emph{near} and when they are \emph{strongly near}.  This article also introduces a kind of \emph{hit and far-miss} topology on the hyperspace of all non-empty closed subsets $\CL(X)$ of a topological space $X$.

\section{Preliminaries}
Recall how a \textit{Lodato proximity} is defined~\cite{Lodato1962,Lodato1964,Lodato1966} (see, also, \cite{Naimpally2009,Naimpally1970}).

\begin{definition} 
Let $X$ be a nonempty set. A \textit{Lodato proximity $\delta$} is a relation on $\mathscr{P}(X)$ which satisfies the following properties for all subsets $A, B, C $ of $X$ :
\begin{itemize}
\item[P0)] $A\ \delta\ B \Rightarrow B\ \delta\ A$
\item[P1)] $A\ \delta\ B \Rightarrow A \neq \emptyset $ and $B \neq \emptyset $
\item[P2)] $A \cap B \neq \emptyset \Rightarrow  A\ \delta\ B$
\item[P3)] $A\ \delta\ (B \cup C) \Leftrightarrow A\ \delta\ B $ or $A\ \delta\ C$
\item[P4)] $A\ \delta\ B$ and $\{b\}\ \delta\ C$ for each $b \in B \ \Rightarrow A\ \delta\ C$
\end{itemize}
Further $\delta$ is \textit{separated }, if 
\begin{itemize}
\item[P5)] $\{x\}\ \delta\ \{y\} \Rightarrow x = y$.
\end{itemize}
\end{definition}

\noindent When we write $A\ \delta\ B$, we read "$A$ is near to $B$", while when we write $A \not \delta B$ we read "$A$ is far from $B$".
A \emph{basic proximity} is one that satisfies the \v{C}ech axioms $P0)-P3)$~\cite[\S 2.5, p. 439]{Cech1966}.
\textit{Lodato proximity} or \textit{LO-proximity} is one of the simplest proximities. We can associate a topology with the space $(X, \delta)$ by considering as closed sets the ones that coincide with their own closure, where for a subset $A$ we have
\[
\mbox{cl} A = \{ x \in X: x\ \delta\ A\}.
\]
This is possible because of the correspondence of Lodato axioms with the well-known Kuratowski closure axioms~\cite[Vol. 1,\S 4.1, 20-21]{Kuratowski1958}. 

By considering the gap between two sets in a metric space ( $d(A,B) = \inf \{d(a,b): a \in A, b \in B\}$ or $\infty$ if $A$ or $B$ is empty ), Efremovi\v c introduced a stronger proximity called \textit{Efremovi\v c proximity} or \textit{EF-proximity}~\cite{Efremovich1951,Efremovich1952}.  

\begin{definition}
An \emph{EF-proximity} is a relation on $\mathscr{P}(X)$ which satisfies $P0)$ through $P3)$ and in addition 
\[A \not\delta B \Rightarrow \exists E \subset X \hbox{ such that } A \not\delta E \hbox{ and } X\setminus E \not\delta B \hbox{ EF-property.}\]
\end{definition}

A topological space has a compatible EF-proximity if and only if it is a Tychonoff space~\cite[\S 5.3]{Naimpally2013} (see, also, S. Willard on Tychonoff spaces~\cite[\S 8, 52-53]{Willard1970}).   

Any proximity $\delta$ on $X$ induces a binary relation over the powerset exp $X,$ usually denoted as $\ll_\delta$ and  named  the  {\it natural strong inclusion associated with } $\delta,$ by declaring that $ A$ is {\it strongly included} in $B, \ A \ll_{\delta} B,$ when $A$ is far from the complement of $B,\mbox{\emph{i.e.}}, A \not\delta X\setminus B.$

By strong inclusion, the \textit{Efremovi\v c property} for $ \delta $ can be written also as a betweenness property   \

 \centerline {  (EF) \  \   \  \  If $A \ll_{\delta} B,$  then there exists some $C$ such that $A \ll_{\delta} \  C \ll_{\delta} \ B$.} \  \ \

A pivotal example of \emph{EF-proximity} is the \emph{metric proximity} in a metric space $(X,d)$ defined by 
\[
A\ \delta\ B \Leftrightarrow d(A,B) = 0.
\]
That is, $A$ and $B$  {\it either intersect or are asymptotic}: for each natural number $n$ there  is a  point  $a_n$ in $A $ and a point  $b_n$ in $B$ such that $d(a_{n},b_{n}) < \frac{1}{n}$.\\

\subsection{Hit and far-miss topologies}
Let $\CL(X)$ be  the hyperspace of all non-empty closed subsets of  a space  $X.$
 {\it Hit and miss} and {\it hit and far-miss}  topologies on $\CL(X)$  are obtained by the join of two halves. Well-known examples are Vietoris topology~\cite{Vietoris1921,Vietoris1922,Vietoris1923,Vietoris1927} (see, also,~\cite{Beer1993,Beer1993hit,DiConcilio2013action,DiConcilio2000SetOpen,DiConcilio2000PartialMaps,DiConcilio1989,DiMaio2008hypertop,DiMaio1995hypertop,DiMaio1992hypertop,Som2008hit,Som2007fnspace,Lucchetti1994,Lucchetti1995,Som2006hypertopology,Guadagni2015}) and Fell topology~\cite{Fell1962HausdorfTop}.  In this article, we concentrate on an extension of Vietoris based on the strongly far proximity. \\
 
\noindent \underline{\emph{Vietoris topology}}\\
\smallskip
Let $X$ be an Hausdorff space. The \textit{Vietoris topology} on $\CL(X)$ has as subbase all sets of the form
\begin{itemize}
\item $V^- = \{E \in \CL(X): E \cap V \neq \emptyset\}$, where $V$ is an open subset of $X$,
\item $W^+ = \{C \in \CL(X): C \subset W \}$, where $W$ is an open subset of $X$.
\end{itemize}

The topology ${\tau_V}^-$ generated by the sets of the first form is called \textbf{hit part} because, in some sense, the closed sets in this family hit the open sets $V$. Instead, the topology 
${\tau_V}^+$ generated by the sets of the second form is called \textbf{miss part}, because the closed sets here miss the closed sets of the form $X \setminus W$.

The Vietoris topology is the join of the two part: $\tau_V = {\tau_V}^- \vee {\tau_V}^+$. It represents the prototype of hit and miss topologies.

The Vietoris topology was modified by Fell. He left the hit part unchanged and in the miss part, ${\tau_F}^+$ instead of taking all open sets $W$, he took only open subsets with compact complement.\\
\underline{\emph{Fell topology:}}  \ \  \  \  $\qquad \  \qquad  \ \ \   \tau_F = {\tau_V}^- \vee {\tau_F}^+$.\\

It is possible to consider several generalizations. For example, instead of taking open subsets with compact complement, for the miss part we can look at subsets running in a family of closed sets $\mathscr{B}$. So we define the {\it hit and miss topology on $\CL(X)$ associated with} ${\mathscr{B}}$ as the topology generated by the join of the hit sets $ A^{-},$  where  $A$ runs over all  open subsets of $X$, with the miss sets $A^{+}$, where   $A$ is once again an open subset of $X,$ but more,  whose  complement runs in   $\mathscr{B}$.

Another kind of generalization concerns the substitution of the inclusion present in the miss part with a strong inclusion associated with a proximity. Namely, when the space $X$ carries a proximity $\delta,$ then a proximity variation of the miss part can be displayed by replacing the miss sets with {\it far-miss sets} \ $A^{++} :=  \{ \ E \in \CL(X) : E \ll_{\delta}   A  \ \}.$

Also, in this case we can consider $A$ with the complement running in a family $\mathscr{B}$ of closed subsets of $X$. Then the {\it hit and far-miss topology , $\tau_{\delta,\mathscr{B}}$, associated with $\mathscr{B}$ } is generated by the join of the hit sets  $A^{-},$   where $A$ is open,   with far-miss sets $A^{++},$ where the complement of $A$ is in $\mathscr{B}$.

Fell topology can be considered as well as an example of hit and far-miss topology. In fact, in any EF-proximity, when a compact set is contained in an open set, it is also strongly contained.   

\setlength{\intextsep}{0pt}
\begin{wrapfigure}[9 ]{R}{0.45\textwidth}
\begin{minipage}{4.5 cm}
\begin{center}
\begin{pspicture}
 (0.0,1.5)(2.5,3.5)
\psframe[linecolor=black](-0.8,0.5)(4.5,3.0)
\pscircle[linestyle=dotted,dotsep=0.05,linecolor=black,linewidth=0.05,style=Torange](0.18,1.55){0.88}
\pscircle[linestyle=dotted,dotsep=0.05,linecolor=black,linewidth=0.05,style=Tyellow](0.98,1.35){0.78}
\pscircle[linecolor=black,linestyle=solid,linewidth=0.05,style=Tgray](3.38,1.85){1.00}
\pscircle[linestyle=dotted,dotsep=0.05,linecolor=black,linewidth=0.05,style=Tyellow](2.38,2.25){0.70}
\rput(-0.5,2.8){\footnotesize  $\boldsymbol{X}$}
\rput(1.40,1.35){\footnotesize $\boldsymbol{\Int A}$}
\rput(0.08,2.15){\footnotesize $\boldsymbol{\Int B}$}
\rput(3.68,1.65){\footnotesize $\boldsymbol{E}$}
\rput(2.08,2.35){\footnotesize  $\boldsymbol{\Int D}$}
\rput(0.28,0.15){\qquad\qquad\qquad\qquad\footnotesize 
                 $\boldsymbol{\mbox{Fig.}\ 3.1.\  A\ \sn\ B,\  D\ \sn\ E}$}
 \end{pspicture}
\end{center}
\end{minipage}
\end{wrapfigure}
\setlength{\intextsep}{2pt}

\section{Main Results}
When we consider a proximity $\delta$, it can happen that $A \delta B$ and they have empty intersection or, if $X$ is a topological space, there could be only one point in the intersection of their closures. We need something more. We want to talk about \emph{strong nearness}~\cite{Peters2015visibility} (see, also,~\cite{Peters2015VoronoiAMSJ,Peters2014KleePhelps,Peters2012notices}) for subsets that at least have some points in common. For this reason, we introduce a new kind of proximity, that we call \emph{almost proximity}.

\begin{definition}
Let $X$ be a topological space. We say that the relation $\sn$ on $\mathscr{P}(X)$ is a strongly near proximity called an \emph{almost proximity}, provided it satisfies the following axioms.
Let $A, B, C \subset X$ and $x \in X$:
\begin{itemize}
\item[N0)] $\emptyset \not\sn A, \forall A \subset X $, and \ $X \sn A, \forall A \subset X$
\item[N1)] $A \sn B \Leftrightarrow B \sn A$
\item[N2)] $A \sn B \Rightarrow A \cap B \neq \emptyset$
\item[N3)] If $\Int(B)$ and $\Int(C)$ are not equal to the empty set, $A \sn B$ or $A \sn C \ \Rightarrow \ A \sn (B \cup C)$
\item[N4)] $\mbox{int}A \cap \mbox{int} B \neq \emptyset \Rightarrow A \sn B$ \qquad \textcolor{blue}{$\blacksquare$}
\end{itemize}
\end{definition}

\noindent When we write $A \sn B$, we read $A$ is \emph{strongly near} $B$.

\begin{example}\label{ex1} \rm
A simple example of \emph{almost proximity} that is also a \emph{Lodato proximity} for nonempty sets $A,B\subset X$ is $A \sn B \Leftrightarrow A \cap B \neq \emptyset$. \qquad \textcolor{blue}{$\blacksquare$}
\end{example}

\begin{example}\label{ex2} \rm
Another example of \emph{almost proximity} is given for nonempty sets $A,B\subset X$ in Fig.~3.1, where $A \sn B \Leftrightarrow \Int (A) \cap\ \Int (B) \neq \emptyset$.  This is not a usual proximity. \qquad \textcolor{blue}{$\blacksquare$}
\end{example}

\begin{example}\label{ex3} \rm
We can define another  \emph{almost proximity} for nonempty sets $D,E\subset X$ in Fig.~3.1 in the following way: $D \sn D \Leftrightarrow E \cap\ \Int (D) \neq \emptyset$ or $\Int (E) \cap D \neq \emptyset$. \qquad \textcolor{blue}{$\blacksquare$}
\end{example}

\noindent Furthermore, for each \emph{almost proximity}, we assume the following relation:
\begin{itemize}
\item[N5)] $x \in \Int (A) \Rightarrow \{x\} \sn A$
\item[N6)] $\{x\} \sn \{y\} \Leftrightarrow x=y$  \qquad \textcolor{blue}{$\blacksquare$}
\end{itemize}
So, for example, if we take the almost proximity of example \ref{ex2}, we have that 
\[
A \sn B \Leftrightarrow \Int A \cap \Int B \neq \emptyset,
\] 
provided $A$ and $B$ are not singletons and if $A = \{x\}$, then $x \in \Int(B)$ and if $B$ is also a singleton, then $x=y$. It turns out that if $A \subset X$ is an open set, then each point that belongs to $A$ is strongly near $A$.\\
 
\medskip

Now we want to define a new kind of \emph{hit and far-miss topology} on $\CL(X)$ by using \emph{almost proximities}. 
That is, let $\tau_\delta$ be the hit and far-miss topology having as subbase the sets of the form:
\begin{itemize}
\item $V^- = \{E \in \CL(X): E \cap V \neq \emptyset\}$, where $V$ is an open subset of $X$,
\item $A^{++} =  \{ \ E \in \CL(X) : E \not\delta X\setminus  A  \ \}$, where $A$ is an open subset of $X$.
\end{itemize}

\noindent where $\delta$ is a Lodato proximity compatible with the topology on $X$,  and $\tau^\doublewedge$ the \emph{strongly hit and far-miss} topology having as subbase the sets of the form:
\begin{itemize}
\item $V^{\doublewedge} = \{E \in \CL(X): E \sn V \}$, where $V$ is an open subset of $X$,
\item $A^{++} =  \{ \ E \in \CL(X) : E \not\delta X\setminus  A  \ \}$, where $A$ is an open subset of $X$.
\end{itemize}

It is possible to define \emph{strongly hit and miss} topology in many other ways by simply mixing strongly hit sets with miss sets or also strongly miss sets in a manner similar to that given in~\cite{PetersGuadagni2015stronglyFar}.

\begin{theorem}
If the space $X$ is $T_1$, the \emph{strongly hit and far-miss} topology is an admissible topology.
\end{theorem}
\begin{proof}
To achieve the result we need a $T_1$ space $X$ in order to work with closed singletons. The result about the far-miss part is already known. So look at the strongly hit part. We want that the canonical injection $i: X \rightarrow \CL(X)$, defined by $i(x)=\{x\}$, is a homeomorphism. It is continuous because $i^{-1}(V^{\doublewedge})= \{x \in X : x \sn V\}=V$. Moreover it is open in the relative topology on $i(X)$ because, for each open set $A$, $i(A)=A^{\doublewedge} \cap i(X)$.
\end{proof}

\begin{remark}
Observe that the usual hit topology is a special strongly hit topology when taking as \emph{almost proximity} either the one in example \ref{ex1} or the one in example~\ref{ex3}.  \qquad \textcolor{blue}{$\blacksquare$}
\end{remark}

Now we consider comparisons between usual \emph{hit and far miss topologies} and \emph{strongly hit and far miss} ones. To this purpose we need the following lemma.

\begin{lemma}\label{lem}
Let $A$ and $H$ be open subsets of a $T_1$ topological space $X$. Then $A^- \subseteq H^{\doublewedge}$ implies $A \subseteq H$.
\end{lemma}
\begin{proof}
By contradiction, suppose $A \not\subseteq H$. Then there exists $a \in A$ such that $a \not\in H$. Hence $\{a\} \cap A \neq \emptyset$ while ${a} \not\sn H $ by property $(N2)$. This is absurd.
\end{proof}

\setlength{\intextsep}{0pt}
\begin{wrapfigure}[9 ]{R}{0.48\textwidth}
\begin{minipage}{4.9 cm}
\begin{center}
\begin{pspicture}
 (0.0,1.5)(2.5,3.5)
\psframe[linecolor=black](-1.2,0.3)(4.8,3.55)
\pscircle[linecolor=black,linestyle=solid,dotsep=0.05,linewidth=0.05,style=Tgreen](0.52,1.95){1.51}
\pscircle[linestyle=solid,dotsep=0.05,linecolor=black,linewidth=0.05,style=Tgray](0.52,1.95){1.08}
\pscircle[linestyle=solid,dotsep=0.05,linecolor=black,linewidth=0.05,style=Tyellow](0.52,1.95){0.68}
\psellipse[linecolor=black,linestyle=dotted,dotsep=0.05,linewidth=0.05,style=Tyellow](2.75,2.25)(1.45,0.50)
\psellipse[linestyle=dotted,dotsep=0.05,linecolor=black,linewidth=0.05,style=Tgray](2.75,2.25)(1.52,1.10)
\psline[linestyle=solid,linecolor=black,linewidth=0.05]{-}(0.52,1.95)(2.12,1.95)
\psline[linestyle=solid,linecolor=black,linewidth=0.05]{-}(0.52,1.95)(0.52,3.50)
\rput(-0.8,3.3){\footnotesize  $\boldsymbol{X}$}
\rput(0.34,1.95){\tiny $\boldsymbol{O}$}
\rput(0.52,1.65){\tiny $\boldsymbol{B(O,1)}$}
\rput(0.41,3.10){\footnotesize $\boldsymbol{s}$}
\rput(0.52,1.15){\tiny $\boldsymbol{S^1(O,s)}$}
\rput(0.52,0.65){\footnotesize $\boldsymbol{\cl E}$}
\rput(2.88,2.25){\tiny  $\boldsymbol{\Int A}$}
\rput(2.88,1.55){\tiny  $\boldsymbol{\Int H}$}
\rput(0.28,0.00){\qquad\qquad\qquad\qquad\footnotesize 
                 $\boldsymbol{\mbox{Fig.}\ 3.2.\  (C = B(O,1)\cup S^1(O,s))\not{\in}H^\doublewedge}$}
 \end{pspicture}
\end{center}
\end{minipage}
\end{wrapfigure}
\setlength{\intextsep}{2pt}

Now let ${\tau_\delta}^{-}$ be the hypertopology having as subbase the sets of the form $V^{-}$, where $V$ is an open subset of $X$, and let $({\tau^\doublewedge})^-$ the hypertopology having as subbase the sets of the form $V^\doublewedge$, again with $V$ an open subset of $X$.

\begin{theorem}
Let $X$ be a $T_1$ topological space. The hypertopologies ${\tau_\delta}^{-}$ and $({\tau^\doublewedge})^-$ are not comparable.
\end{theorem}

\begin{proof}
First, we prove that ${\tau_\delta}^{-} \not\subset({\tau^\doublewedge})^-$. Take $A^- \in {\tau_\delta}^{-}$ and $E \in A^-$, with $E \in \CL(X)$. We ask whether there exists an open set $H$ such that $E \in H^\doublewedge \subset A^-$. To this purpose, consider the almost proximity of example \ref{ex2} and let $E$ be such that $\Int(E)= \emptyset$. Then $E \cap A \neq \emptyset$, but $\Int(E) \cap H = \emptyset$ for each $H \subset X$.

Conversely, we want to prove that $({\tau^\doublewedge})^- \not\subset {\tau_\delta}^{-}$. Suppose $X = \mathbb{R}^2$ and consider, as before, the almost proximity of example $\ref{ex2}$. Take the closed set $E=B(O,2)$, that is the closed ball with the origin as center and radius 2, and the open set $H$ as in Fig. 2.2. So $E \sn H$. By lemma \ref{lem}, to identify an open set $A$ such that $E \in A^- \subseteq H^\doublewedge$, we need to consider $A \subseteq H$. But now it is possible to choose $C \in A^-$ such that $C \not\in H^\doublewedge$, for each $A \subset H$. For example, we can take $C = B(O,1) \cup S^1(O,s)$, where $S^1(O,s)$ is the circumference with the origin as center and radius $s$ for a suitable value of $s$ (see, {\em e.g.}, Fig. 3.2). \\

\end{proof}


\end{document}